\documentclass[11pt,a4paper]{article}
\usepackage{amsmath}
\usepackage{amssymb}
\usepackage{mathrsfs}
\usepackage{hyperref}
\numberwithin{equation}{section}

\evensidemargin=\oddsidemargin\textwidth=142mm \textwidth=160mm
\newtheorem{theorem}{Theorem}[section]
\newtheorem{lemma}{Lemma}[section]

\newtheorem{conjecture}{Conjecture}[section]

\newenvironment{proof}[1][Proof]{\begin{trivlist}
\item[\hskip \labelsep {\bfseries #1}]}{\end{trivlist}}

\begin{document}
\title {A Derivative-Hilbert operator acting on Hardy spaces\footnote{ The research was
supported by
	the National Natural Science Foundation of China (Grant No. 11671357)}}
\author{Shanli Ye\footnote{Corresponding author. E-mail address:  slye@zust.edu.cn}
\quad\quad Guanghao Feng\footnote{E-mail address:gh945917454@foxmail.com}     \\
(\small \it School of Sciences, Zhejiang University of Science and Technology, Hangzhou
310023, China, P. R. China$)$}

 \date{}
\maketitle
\begin{abstract}

Let $\mu$ be a positive Borel measure on the interval $[0,1)$. The Hankel matrix
$\mathcal{H}_\mu= (\mu_{n,k})_{n,k\geq0}$ with entries $\mu_{n,k}= \mu_{n+k}$, where
$\mu_n=\int_{ [0,1)}t^nd\mu(t)$, induces formally the operator
$$\mathcal{DH}_\mu(f)(z)=\sum_{n=0}^\infty (\sum_{k=0}^\infty \mu_{n,k}a_k)(n+1)z^n$$
on the space of all analytic function $f(z)=\sum_{k=0}^ \infty a_k z^n$   in the unit disc
$\mathbb{D}$.
We characterize those positive Borel  measures on $[0,1)$ such that
$\mathcal{DH}_\mu(f)(z)= \int_{[0,1)} \frac{f(t)}{{(1-tz)^2}} d\mu(t)$ for all in Hardy
spaces $H^p(0<p<\infty)$, and among them we describe those for which  $\mathcal{DH}_\mu$ is
a bounded(resp.,compact) operator from $H^p(0<p <\infty)$ into $H^q(q > p$ and $q\geq 1$).
 We also study the analogous problem in Hardy spaces $H^p(1\leq p\leq 2)$.
\\
{\small\bf Keywords}\quad {Derivative-Hilbert operator, Hardy spaces, Carleson measure
 \\
    {\small\bf 2020 MR Subject Classification }\quad 47B35, 30H20, 30H30\\}
\end{abstract}
\maketitle

\section{Introduction}\label{s1}  %注入一级标题示例,大括号内替换文章的一级标题

Let $\mu$ be a positive Borel measure on the interval $[0,1)$. The Hankel matrix
$\mathcal{H}_\mu= (\mu_{n,k})_{n,k\geq0}$ with entries $\mu_{n,k}= \mu_{n+k}$, where
$\mu_n=\int_{ [0,1)}t^nd\mu(t)$, For analytic functions $f(z)=\sum_{k=0}^ \infty a_k z^n$ ,
the generalized Hilbert operator define as
$$\mathcal{H}_\mu(f)(z)=\sum_{n=0}^\infty (\sum_{k=0}^\infty \mu_{n,k}a_k)z^n.$$
whenever the right hand side makes sense and defines an analytic function in $\mathbb{D}$.

In recent decades, the generalized Hilbert operator $\mathcal{H}_\mu$ which induced by the
Hankel matrix $\mathcal{H}_\mu$ have been studied extensively. For example, Galanopoulos and
Pel\'{a}ez \cite{9} characterized the Borel measure $\mu$ for which the Hankel operator
$\mathcal{H}_\mu$ is a bounded (resp.,compact) operator on $H^1$. Then Chatzifountas, Girela
and Pel\'{a}ez \cite{3} extended this work the all Hardy spaces $H^p$ with $0<p< \infty$. In
\cite{10}, Girela and Merch\'{a}n studied the operator acting on certain conformally invariant
spaces of analytic functions on the disk.

In 2021, Ye and Zhou \cite{18} firstly used the Hankel matrix defined the Derivative-Hilbert
operator $\mathcal{DH}_\mu$ as
\begin{align}\label{1.1}
\mathcal{DH}_\mu(f)(z)=\sum_{n=0}^\infty (\sum_{k=0}^\infty \mu_{n,k}a_k)(n+1)z^n.
\end{align}

Another generalized Hilbert-integral operator related to $\mathcal{DH}_\mu$ denoted by
$\mathcal{I}_{\mu_\alpha} (\alpha \in \mathbb{N}^+$) is defined by
$$ \mathcal{I}\mu_\alpha(f)(z)=\int_{[0,1)} \frac{f(t)}{{(1-tz)^\alpha}} d\mu(t).$$
whenever the right hand side makes sense and defines an analytic function in $\mathbb{D}$.
We can easily see that the case $\alpha=1$ is the integral representation of the generalized
Hilbert operator. Ye and Zhou characterized the measure $\mu$ for which $ \mathcal{I}\mu_2$
and $\mathcal{DH}_\mu$ is bounded (resp.,compact) on Bloch space \cite{18} and Bergman spaces
\cite{17}.

In this paper, we consider the operators
\begin{align}
   & \mathcal{DH}_\mu,\mathcal{I}_{\mu_2}: H^p\rightarrow H^q,  \quad 0 <p< \infty, \quad
   q\geq p
   \; and \; q\geq 1. \notag \\
   & \mathcal{DH}_\mu,\mathcal{I}_{\mu_2}: H^p\rightarrow B_q,  \quad 0<p \leq 1, \quad
   0<q<1. \notag
\end{align}
The aim is to study the boundedness (resp.,compactness) of $\mathcal{I}_{\mu_2}$ and
$\mathcal{DH}_\mu$.

In this article we characterize the positive Borel measures $\mu$ for which the operator
which $ \mathcal{I}_{\mu_2}$ and $\mathcal{DH}_\mu$ is well defined in the Hardy spaces
$H^p$. Then we give the necessary and sufficient conditions  such that operator
$\mathcal{DH}_\mu$ is bounded (resp.,compact) from the Hardy space $H^p(0<p< \infty)$ into
the space $H^q (q \geq p$ and $q \geq 1) $, or from $H^p(0 < p \leq 1)$ into $B_q(0<q<1)$.

\section{Notation and Preliminaries}\label{s2}
 Let $\mathbb{D}$ denote the open unit disk of the complex plane, and let $H(\mathbb{D})$ denote the set of all analytic functions in $\mathbb{D}$.

If $0<r<1$ and $f\in H(\mathbb{D})$, we set
\begin{align}
   & M_p (r,f)=\left( \frac{1}{2\pi} \int_0^{2\pi} |f(re^{i \theta})|^p d\theta \right)^\frac{1}{p}, \quad 0<p<\infty.  \notag\\
   & M_\infty (r,f)=\sup_{|z|=r}|f(z)|.\notag
\end{align}

For $0<p\leq \infty$, the Hardy space $H^p$ consists of those $f \in H(\mathbb{D})$ such that
$$||f||_{H^p} \overset{def}{=} \sup_{0<r<1}M_p(r,f)<\infty.$$

We refer to \cite{6} for the notation and results regarding Hardy spaces.

For $0 < q < 1$, we let $B_q$ \cite{22} denote the space consisting of those $f \in H(\mathbb{D})$ for which
$$||f||_{B_q}=\int_0^1 (1-r)^{1/q-2}M_1(r,f)dr <\infty.$$
The Banach space $B_q$ is the ``containing Banach space" of $H^q$, that is, $H^q$ is a dense subspace of $B_q$, and the two spaces have the same continuous linear functionals. (We mention \cite{22} as general references for the $B_q$ spaces.)

The space $BMOA$ consists of those function $f \in H^1$ whose boundary values have bounded mean oscillation on $\partial \mathbb{D}$ as defined by John and Niirenberg.  There are many characterizations of $BMOA$ functions. Let us mention the following.

 Let $\varphi_a(z)=\frac{a-z}{1-\overline{a}z}$ be a M$\ddot{o}$bius transformations. If $f$ is an analytic function in $\mathbb{D}$, then $f \in BMOA $ if and only if
$$||f||_{BMOA} \overset{def}{=} |f(0)|+||f||_*< \infty.$$
where
$$||f||_* \overset{def}{=}\sup_{a\in \mathbb{D}}||f\circ\varphi_a -f(a)||_{H^2}.$$
It is clear that the seminorm $||\;||_*$ is conformally invariant. If
$$\lim_{|a|\rightarrow 1}||f\circ\varphi_a -f(a)||_{H^2}=0,$$
then we say that $f$ belongs to the space VMOA (analytic functions of vanishing mean oscillation). We refer to \cite{19} for the theory of $BMOA$ functions.

Finally, we recall that a function $f\in H(\mathbb{D})$ is said to be a Bloch function if
$$||f||_\mathcal{B}\overset{def}{=} |f(0)|+\sup_{z\in \mathbb{D}}(1-|z|^2)|f'(z)|< \infty.$$
The space of all Bloch functions is denoted by $\mathcal{B}$. A classical reference for the theory of Bloch functions is \cite{11,14}. The relation between these spaces we introduced above is well known,
$$H^\infty \varsubsetneq BMOA \varsubsetneq \mathcal{B}, \quad BMOA \varsubsetneq \bigcap_{0<p<\infty} H^p.$$

Let us recall the knowledge of Carleson measure, which is a very useful tool in the study of Banach spaces of analytic functions. For $0<s<\infty$, a positive Borel measure $\mu$ on $\mathbb{D}$ will be called an $s$-Carleson measure if there exists a positive constant $C$ such that
$$\mu(S(I))\leq C|I|^s .$$
The Carleson square $S(I)$ is defined as
$$S(I)=\{z=re^{i\theta}:e^{i\theta}\in I; 1-\frac{|I|}{2\pi}\leq r \leq 1\}.$$
where $I$ is an interval of $\partial \mathbb{D}$ , $|I|$ denotes the length of $I$. If $\mu $ satisfies $\lim_{|I|\rightarrow 0} \frac{\mu(S(I)) }{|I|^s}=0$, we call $\mu$ is a vanishing $s$-Carleson measure.

Let $\mu$ be a positive Borel measure on $\mathbb{D}$. For $0\leq \alpha < \infty$ and $ 0<s< \infty $ we say that $\mu$ is $\alpha$-logarithmic $s$-Carleson measure,if there exists a positive constant $C$ such that
 $$\frac{\mu(S(I))(\log\frac{ 2\pi }{|I|})^\alpha}{|I|^s} \leq C \quad  \quad I \subset \partial \mathbb{D}.$$
If $\mu(S(I))(\log\frac{ 2\pi }{|I|})^\alpha=o(|I|^s)$ , as  $|I|\rightarrow 0$, we say that $\mu$ is vanishing $\alpha$-logarithmic $s$-Carleson measure\cite{13,15}.

Suppose $\mu$ is a $s$-Carleson measure on $\mathbb{D}$, we see that the identity mapping $i$ is well defined from $H^p$ into $L^q(\mathbb{D},\mu).$ Let $\mathcal{N}(\mu)$ be the norm of $i$. For $0<r<1$, let
$$d\mu_r(z)=\chi_{r<|z|<1}(t)d\mu(t).$$
Then $\mu$ is a vanishing $s$-Carleson measure if and only if
\begin{align}\label{2.1}
  \mathcal{N}(\mu_r) \rightarrow 0 \quad as  \quad r\rightarrow 1^-.
\end{align}

A positive Borel measure on [0, 1) also can be seen as a Borel measure on $\mathbb{D}$ by identifying it with the measure $\mu$ defined by
$$\tilde{\mu}(E)=\mu(E\bigcap [0,1)).$$
for any Borel subset E of $\mathbb{D}$. Then a positive Borel measure $\mu$ on [0,1) can be seen as an $s$-Carleson measure on $\mathbb{D}$, if
$$\mu([t,1)) \leq C(1-t)^s,  \quad 0 \leq t<1 .$$
Also, we have similar statements for vanishing $s$-Carleson measures, $\alpha$-logarithmic
$s$-Carleson and vanishing $\alpha$-logarithmic $s$-Carleson measures.

As usual, throughout this paper, $C$ denotes a positive constant which depends only on the displayed parameters but not necessarily the same from one occurrence to the next. For any given $p > 1$, $p'$ will denote the conjugate index of p, that is, $1/p + 1/p' = 1$.

\section{Conditions such that $\mathcal{DH}_\mu$ is well defined on Hardy spaces}\label{s3}
In this section, we find the sufficient condition such that $\mathcal{DH}_\mu$ are well defined
in $H^p(0<p< \infty)$ and obtain that $\mathcal{DH}_\mu(f)= \mathcal{I}_{\mu_2}(f)$, for all $f\in H^p$, with the certain condition.

We first recall two results about the coefficients of functions in Hardy spaces.
\begin{lemma}{\cite[p.98]{6}}\label{Le3.1} If
$$f(z)=\sum_{n=0}^\infty a_n z^n \in H^p, \quad 0<p\leq 1,$$
then
$$a_n=o(n^{1/p-1}),$$
and
$$ |a_n|\leq Cn ^{1/p-1}||f||_{H^p}.$$
\end{lemma}

\begin{lemma}{\cite[p.95]{6}}\label{Le3.2} If
$$f(z)=\sum_{n=0}^\infty a_n z^n \in H^p, \quad 0<p\leq 2,$$
then $\sum n^{p-2}|a_n|^p < \infty$ and
$$ \left\{ \sum_{n=0}^\infty(n+1)^{p-2}|a_n|^p \right\}^{1/p} \leq C ||f||^p_{H^p}.$$

\end{lemma}

\begin{theorem}\label{Th3.1}
Suppose $0 < p < \infty$ and let $\mu$ be a positive Borel measure on $[0,1)$. Then the power series in (\ref{1.1}) defines a well defined analytic function in $\mathbb{D}$ for every $f \in H^p$ in any of the two following cases.
~\\
\item (\romannumeral1) The measure $\mu$ is a 1/p-Carleson measure, if $0<p\leq 1.$
\item (\romannumeral2) The measure $\mu$ is a 1-Carleson measure, if $1<p< \infty.$
\end{theorem}

Furthermore, in such cases we have that
\begin{align}\label{3.1}
  \mathcal{DH}_\mu(f)(z)=\int_{[0,1)} \frac{f(t)}{{(1-tz)^2}} d\mu(t)=\mathcal{I}\mu_2(f)(z).
\end{align}

\begin{proof}
First recall a well known result of Hastings \cite{12}: For $0<p\leq q< \infty$, $\mu$ is a $q/p$-Carleson measure if and only if there exists a positive constant $C$ such that

\begin{align}\label{3.2}
  \left\{\int_{[0,1)} |f(t)|^q d\mu(t) \right\}^\frac{1}{q} \leq C||f||_{H^p}, \quad  for \; all \; f \in H^p.
\end{align}

(\romannumeral1) Suppose that $0<p \leq 1$ and $\mu$ is a $1/p$-Carleson measure. Then (\ref{3.2}) gives
$$\int_{[0,1)} |f(t)| d\mu(t)  \leq C||f||_{H^p}, \quad  for \; all \; f \in H^p.$$
~\\
Fix $f(z) = \sum _{k=0}^ \infty a_k z^k \in H^p$ and $z$ with $|z|<r$, $0 <r< 1$. It follows that

\begin{align}
\int_{[0,1)} \frac{|f(t)|}{{|1-tz|^2}} d\mu(t) & \leq \frac{1}{(1-r)^2} \int_{[0,1)}|f(t)|d\mu (t)    \notag \\
& \leq C \frac{1}{(1-r)^2}||f||_{H^p} .  \notag
\end{align}
This implies that the integral $\displaystyle\int_{[0,1)} \frac{f(t)}{{(1-tz)^2}} d\mu(t)$ uniformly converges and that
\begin{align}\label{3.3}
 \mathcal{I}\mu_2(f)(z)= &\int_{[0,1)} \frac{f(t)}{{(1-tz)^2}} d\mu(t) \notag \\
  =& \sum_{n=0}^\infty(n+1)\left(\int_{[0,1)}t^n f(t)d\mu (t)\right)z^n.
\end{align}

Take $f(z) = \sum _{k=0}^ \infty a_k z^k \in H^p$. Since $\mu$ is $1/p$-Carleson measure, by \cite[Proposition 1]{3} and Lemma \ref{Le3.1},  we have that there exists $C>0$ such that
$$|\mu_{n,k}|=|\mu_{n+k}|\leq \frac{C}{(k+1)^\frac{1}{p}},$$
$$|a_k| \leq C(k+1)^{(1-p)/p} \quad for \; all \; n,k.$$
Then it follows that, for every $n$,
\begin{align}
   (n+1)\sum_{k=0}^\infty |\mu_{n,k}||a_k| & \leq C(n+1) \sum_{k=0}^\infty \frac{|a_k|}{(k+1)^{1/p}} =C(n+1) \sum_{k=0}^\infty \frac{|a_k|^p |a_k|^{1-p}}{(k+1)^{1/p}}\notag \\
   & \leq C(n+1)\sum_{k=0}^\infty \frac{|a_k|^p (k+1)^{(1-p)^2/p}}{(k+1)^{1/p}} \notag\\
   &=C(n+1)\sum_{k=0}^\infty(k+1)^{p-2}|a_k|^p\notag,
\end{align}
and then by Lemma \ref{Le3.2}, we deduce that
$$(n+1)\sum_{k=0}^\infty |\mu_{n,k}||a_k| \leq C(n+1)||f||_{H^p}^p.$$
This implies that $\mathcal{DH}_\mu$ is a well defined for all $z\in \mathbb{D}$ and that
\begin{align}
 \mathcal{DH}_\mu(f)(z)=& \sum_{n=0}^\infty (n+1) (\sum_{k=0}^\infty \mu_{n,k}a_k)z^n \notag\\
  =&  \sum_{n=0}^\infty (n+1) \int_{[0,1)} t^n f(t)d\mu(t) z^n \notag \\
  = & \int_{[0,1)} \frac{f(t)}{{(1-tz)^2}} d\mu(t).
\end{align}
This give that $\mathcal{DH}_\mu(f)=\mathcal{I}_{\mu_2} (f).$
~\\

(\romannumeral2) When $1<p<\infty$, since  $\mu$ is a 1-Carleson measure, (\ref{3.2}) holds, then the argument used in the proof of (\romannumeral1) give that, for every $f \in H^p$, $I_{\mu_2}$ is well defined analytic function in  $\mathbb{D}$ and we have (\ref{3.3}).

And since $\mu$ is $1$-Carleson measure by \cite[Theorm 3]{3}, we know
$$(n+1)\sum_{k=0}^\infty \mu_{n,k}a_k =(n+1)\int_{[0,1)}t^n f(t)d\mu (t), $$
which  implies that $\mathcal{DH}_\mu$ is a well defined for all $z\in \mathbb{D}$, and $\mathcal{DH}_\mu(f)=\mathcal{I}_{\mu_2} (f).$
\end{proof}
\section{ Bounededness of  $\mathcal{DH}_\mu$ acting on Hardy spaces}\label{s4}

In this section, we mainly characterize those measures $\mu$ for which $\mathcal{DH}_\mu$ is a bounded (resp., compact) operator from $H^p$ into $H^q$ for some $p$ and $q$.
\begin{theorem}\label{Th4.1}
Suppose $0<p\leq 1$ and let $\mu$ be a positive Borel measure on $[0,1)$ which satisfies the condition in Theorem \ref{Th3.1}.
~\\
\item (\romannumeral1) If  $q>1$, then $\mathcal{DH}_\mu$ is a bounded operator from $H^p$ into $H^q$ if and only if $\mu$ is a $(1/p+1/q'+1)$-Carleson measure.
\item (\romannumeral2) If  $q = 1$, then $\mathcal{DH}_\mu$ is a bounded operator from $H^p$ into $H^1$ if and only if $\mu$ is a $(1/p+1)$-Carleson measure.
\item (\romannumeral3) If  $0<q<1$,  then $\mathcal{DH}_\mu$ is a bounded operator from $H^p$ into $B_q$ if and only if $\mu$ is a $(1/p+1)$-Carleson measure.
\end{theorem}

\begin{proof}
Suppose $0<p\leq 1$. Since $\mu$ satisfies the condition in Theorem \ref{Th3.1}, as in the proof of Theorem \ref{Th3.1}, we obtain that
$$\int_{[0,1)} |f(z)| d\mu(t)< \infty, \quad for \; all \; f \in H^p.$$
Hence, it follows that
\begin{align}\label{4.1}
  &\int_0^{2\pi} \int_{[0,1)}\left|\frac{f(t)g(e^{i\theta})}{(1-re^{i\theta}t)^2}\right|d\mu(t)d\theta   \notag\\
  & \leq \frac{1}{(1-r)^2}  \int_{[0,1)}|f(t)|d\mu(t) \int_0^{2\pi}|g(e^{i\theta})|d\theta  \notag\\
  & \leq \frac{C ||g||_{H^1}}{(1-r)^2} < \infty \quad  0\leq r<1,\;f\in H^p,g\in H^1.
\end{align}
Using Theorem \ref{Th3.1}, (\ref{4.1}) and Fubini's theorem, and Cauchy's integral representation of of $H^1$ \cite{6}, we obtain
\begin{align}\label{4.2}
    &\frac{1}{2\pi}\int_0^{2\pi}\overline{\mathcal{DH}_\mu(f)(re^{i\theta})}g(e^{i\theta})d\theta \notag \\
    =& \frac{1}{2\pi}\int_0^{2 \pi} \left( \int_{[0,1)} \frac{\overline{f(t)}d\mu(t)}{(1-re^{-i\theta}t)^2}\right)g(e^{i\theta})d\theta \notag  \notag \\
    =&\frac{1}{2\pi}\int_{[0,1)}\overline{f(t)}\int_{|e^{i \theta}|=1}
  \frac{g(e^{i\theta})e^{i\theta}}{(e^{i\theta}-rt)^2}de^{i\theta} d\mu(t) \notag\\
    =&\frac{1}{2\pi}\int_{[0,1)}\overline{f(t)}(g(rt)rt)'d\mu(t)  \notag \\
    =&\frac{1}{2\pi}\int_{[0,1)}\overline{f(t)}(g(rt)+rtg'(rt))d\mu(t),\quad 0\leq r<1,\;f\in H^p,g\in H^1.
\end{align}

(\romannumeral1)First consider  $q>1$. Using (\ref{4.2}) and the  duality theorem \cite{6} for $H^q$ which says that $(H^p)^*\cong H^{p'}$ and $(H^{p'})^*\cong H^p(p>1)$, under the Cauchy pairing
\begin{align}\label{4.9}
  <f,g> = \frac{1}{2\pi} \int_0^{2\pi}\overline{f(e^{i\theta})}g(e^{i\theta})d\theta, \quad f \in H^p, g\in H^{p'}.
\end{align}
We obtain that $\mathcal{DH}_\mu$ is a bounded operator from $H^p$ into $H^q$ if and only if there exists a positive constant $C$ such that
\begin{align}\label{4.3}
  \left|\int_{[0,1)}\overline{f(t)}(g(t)+tg'(t))d\mu(t) \right| \leq C||f||_{H^p}||g||_{H^{q'}},\;f\in H^p,g\in H^{q'}.
\end{align}

Assume that $\mathcal{DH}_\mu$ is a bounded operator from $H^p$ into $H^q$. Take the families of text functions
$$f_a(z)=\left(\frac{1-a^2}{(1-az)^2} \right)^{1/p}, \quad
g_a(z)=\left(\frac{1-a^2}{(1-az)^2} \right)^{1/{q'}}, \quad 0 <a< 1.$$

A calculation shows that $\{f_a\} \subset H^p$, $\{g_a\} \subset H^{q'}$ and
\begin{align}
\sup_{a \in [0,1)}||f||_{H^p} < \infty \quad and \quad \sup_{a \in [0,1)}||g||_{H^{q'}} < \infty , \notag
\end{align}
It follows that
\begin{align}
  \infty & > C \;\sup_{a \in [0,1)}||f||_{H^p}\sup_{a \in [0,1)}||g||_{H^{q'}}  \notag \\
   & \geq \left|\int_{[0,1)}\overline{f_a(t)}(g_a(t)+tg'_a(t))d\mu(t) \right| \notag \\
   & \geq  \int_{[a,1)} \left(\frac{1-a^2}{(1-at)^2} \right)^{1/p} \left(\left(\frac{1-a^2}{(1-at)^2}\right)^{1/{q'}}+\frac{2t^2}{q'}\left(\frac{1-a^2}{(1-at)^{2+{q'}}} \right)^{1/{q'}}  \right) d\mu(t)\notag \\
    & \geq \frac{1}{(1-a^2)^{1/p+1/{q'}+1}} \mu([a,1)).
\end{align}
This is equivalent to saying that $\mu$ is a $(1/p+1/q'+1)$-Carleson measure.

On the other hand, suppose $\mu$ is a $(1/p+1/q'+1)$-Carleson measure,It is well known that any function $g \in H^{q'}$ \cite{6} has the property
\begin{align}\label{4.5}
  |g(z)| \leq C\frac{||g||_{H^{q'}}}{(1-|z|)^{1/q'}}.
\end{align}
By Cauchy formula, we can obtain that
\begin{align}\label{4.7}
  |g'(z)| \leq C \frac{||g||_{H^{q'}}}{(1-|z|)^{1/q'+1}}.
\end{align}
 Let $d\nu(t)=\frac{1}{(1-t)^{1/q'+1}}d\mu(t)$. Using Lemma 3.2 of \cite{20}, we obtain that $\nu$ is a $1/p$-Carleson.
 This together with (\ref{4.5}) and (\ref{4.7}) we obtain that
\begin{align}\label{4.8}
  \left|\int_{[0,1)}\overline{f(t)}(g(t)+tg'(t))d\mu(t) \right| & \leq C ||g||_{H^{q'}}\int_{[0,1)}|f(t)| \left(\frac{1}{(1-t)^{1/q'}}+\frac{t}{(1-t)^{1/q'+1}} \right)d\mu(t) \notag \\
  & \leq C ||g||_{H^{q'}}\int_{[0,1)}|f(t)|d\nu(t),  \notag \\
  & \leq C ||g||_{H^{q'}} ||f||_{H^p}, \quad g\in H^{q'} \;f\in H^p.
\end{align}

Hence (\ref{4.3}) holds and then $\mathcal{DH}_\mu$ is a bounded operator from $H^p$ into $H^q$.

(\romannumeral2)We shall use Fefferman's duality theorem, which says that $(H^1)^*\cong BMOA $ and $(VMOA)^*\cong H^1$, under the Cauchy pairing

\begin{align}\label{4.9}
  <f,g> =\lim_{r\rightarrow 1^-} \frac{1}{2\pi} \int_0^{2\pi}\overline{f(re^{i\theta})}g(e^{i\theta})d\theta, \quad f \in H^1, g\in BMOA(resp.VMOA).
\end{align}
Using the duality theorem and (\ref{4.2}) it follows that $\mathcal{DH}_\mu$ is a bounded operator from $H^p$ into $H^1$ if and only if there exists a positive constant C such that
\begin{align}\label{4.10}
\left|\int_{[0,1)}\overline{f(t)}(g(rt)+rtg'(rt))d\mu(t) \right| \leq C||f||_{H^p}||g||_{BMOA},\notag \\
\quad 0\leq r<1,\;f\in H^p,\;g\in VMOA.
\end{align}
Suppose that $\mathcal{DH}_\mu$ is a bounded operator from $H^p$ into $H^1$. Take the families of text functions
$$f_a(z)=\left(\frac{1-a^2}{(1-az)^2} \right)^{1/p}, \quad
g_a(z)=\log \frac{e}{1-az}, \quad 0 <a< 1.$$
A calculation shows that $\{f_a\} \subset H^p$, $\{g_a\} \subset VMOA$ and
\begin{align}
\sup_{a \in [0,1)}||f||_{H^p} < \infty \quad and \quad \sup_{a \in [0,1)}||g||_{BMOA} < \infty . \notag
\end{align}
We let $r\in [a,1)$, and obtain
\begin{align}
 \infty & > C \;\sup_{a \in [0,1)}||f||_{H^p}\sup_{a \in [0,1)}||g||_{BMOA}  \notag \\
 & \geq \left|\int_{[0,1)}\overline{f_a(t)}(g_a(rt)+rtg'_a(rt))d\mu(t) \right| \notag \\
  & \geq  \int_{[a,1)} \left(\frac{1-a^2}{(1-at)^2} \right)^{1/p} \left(\log \frac{e}{1-art}+\frac{art}{1-art}  \right)d\mu(t) \notag \\
  & \geq  \frac{1}{(1-a^2)^{1/p+1}}\mu([a,1)).
\end{align}
This is equivalent to saying that $\mu$ is a $(1/p+1)$-Carleson measure.

On the other hand, suppose $\mu$ is a $(1/p+1)$-Carleson measure. It is well known that any function $g \in$ $\mathcal{B}$ \cite{14} has the property
\begin{align}\label{4.12}
  |g(z)| \leq C||g||_{\mathcal{B}}\log\frac{e}{1-|z|},\quad and \quad |g'(z)|\leq C\frac{ ||g||_{\mathcal{B}}}{1-|z|} \quad for \;all \; z\in \mathbb{D}.
\end{align}
Let $d\nu(t)=\frac{1}{1-t}d\mu(t)$, then $\nu$ is a $1/p$-Carleson. Using (\ref{4.10}), (\ref{4.12}) and $BMOA \subset \mathcal{B}$ \cite[Theorem 5.1]{19},we obtain that
\begin{align}
  \left|\int_{[0,1)}\overline{f(t)}(g(rt)+rtg'(rt))d\mu(t) \right| & \leq C||g||_{\mathcal{B}} \int_{[0,1)}|f(t)| \left({\log \frac{1}{1-t}}+\frac{t}{1-t} \right)d\mu(t) \notag \\
  & \leq C||g||_{BMOA} \int_{[0,1)}|f(t)|d\nu(t) \notag \\
  & \leq C||g||_{BMOA}||f||_{H^p}, \quad f\in H^p,\;g\in VMOA.
\end{align}
Hence (\ref{4.10}) holds and then $\mathcal{DH}_\mu$ is a bounded operator from $H^p$ into $H^1$.

(\romannumeral3) If  $0<q<1$, and $0<p<1$. Now we recall \cite[Theorem 10]{22} that $B_q$ can be identified with the dual of a certain subspace $X$ of $H^ \infty$ under the pairing
\begin{align}\label{4.171}
  <f,g>=\lim_{r\rightarrow 1^-} \frac{1}{2\pi} \int_0^{2\pi}\overline{f(re^{i\theta})}g(e^{i\theta})d\theta, \quad f \in B_q, \;g\in X.
\end{align}

This together with (\ref{4.2}) and (\ref{4.171}), we obtain that $\mathcal{DH}_\mu$ is a bounded operator from $H^p$ into $B_q$ if and only if there
exists a positive constant C such that
\begin{align}\label{4.181}
  \left|\int_{[0,1)}\overline{f(t)}(g(rt)+rtg'(rt))d\mu(t) \right| \leq C||f||_{H^p}||g||_{H^\infty},\quad 0\leq r<1,\;f\in H^p,g\in H^\infty.
\end{align}

Suppose that $\mathcal{DH}_\mu$ is a bounded operator from $H^p$ into $B_q$, Take the families of text functions
$$f_a(z)=\left(\frac{1-a^2}{(1-az)^2} \right)^{1/p}, \quad
g_a(z)=\frac{1-a^2}{1-az}, \quad 0 <a< 1.$$

A calculation shows that $\{f_a\} \subset H^p$, $\{g_a\} \subset H^\infty$ and
\begin{align}
\sup_{a \in [0,1)}||f||_{H^p} < \infty \quad and \quad \sup_{a \in [0,1)}||g||_{H^\infty} < \infty . \notag
\end{align}
We let $r\in [a,1)$, obtain
\begin{align}
  \infty & > C \;\sup_{a \in [0,1)}||f||_{H^p}\sup_{a \in [0,1)}||g||_{H^\infty}  \notag \\
   & \geq \left|\int_{[0,1)}\overline{f_a(t)}(g_a(rt)+rtg'_a(rt))d\mu(t) \right| \notag \\
   & \geq  \int_{[a,1)} \left(\frac{1-a^2}{(1-at)^2} \right)^{1/p} \left(\left(\frac{1-a^2}{(1-art)}\right)+\left(\frac{art(1-a^2)}{(1-art)^2} \right) \right) d\mu(t)\notag \\
    & \geq \frac{1}{(1-a^2)^{1/p+1}} \mu([a,1)).
\end{align}
This is equivalent to saying that $\mu$ is a $(1/p+1)$-Carleson measure.

On the other hand, suppose $\mu$ is a $(1/p+1)$-Carleson measure, then $d\nu(t)=\frac{1}{1-t}d\mu(t)$ is a $1/p$-Carleson. Then we have that

\begin{align}
  \left| \int_{[0,1)}\overline{f(t)}(g(rt)+rtg'(rt))d\mu(t) \right| & \leq C||g||_{H^\infty} \int_{[0,1)}|f(t)| \left(1+\frac{t}{1-t} \right)d\mu(t)  \notag \\
  & \leq C||g||_{H^\infty} \int_{[0,1)}|f(t)|d\nu(t) \notag \\
  & \leq C||g||_{H^\infty}||f||_{H^p}, \quad f\in H^p,\;g\in H^\infty.
\end{align}
Hence (\ref{4.181}) holds and then $\mathcal{DH}_\mu$ is a bounded operator from $H^p$ into $B_q$.

\end{proof}

Next,we will consider $1<p<\infty,q \geq p$ , and give a sufficient condition, and the other necessary condition for bounded of $\mathcal{DH}_\mu$ from $H^p$ into $H^q$ respectively.

We firstly give Lemma \ref{Le4.1} which is useful when we proof the Theorem \ref{Th4.2}.

\begin{lemma}\label{Le4.1}
Let $\mu$ be a positive measure on $[0,1)$ and $\gamma>0,\alpha>0$. If $\mu$ is a $(\alpha+\gamma)$-Carleson measure, then
$$\int_{[0,1)} \frac{1}{(1-t)^\alpha}d\mu(t)<\infty.$$
\end{lemma}
The result is obvious, we omit the details.

\begin{theorem}\label{Th4.2}
 Let $1<p\leq q<\infty$ and  $\mu$ be a positive Borel measure on $[0,1)$ which satisfies the condition in Theorem \ref{Th3.1}.
\item (\romannumeral1)If $\mu$ is a $(1/p+1/q'+1+\gamma)$-Carleson measure for any  $\gamma>0$, then $\mathcal{DH}_\mu$ is a bounded operator from $H^p$ into $H^q$.
\item (\romannumeral2) If  $\mathcal{DH}_\mu$ is a bounded operator from $H^p$ into $H^q$, then $\mu$ is a $(1/p+1/q'+1)$-Carleson measure.
\end{theorem}

\begin{proof}
Suppose $\mu$ is a $(1/p+1/q'+1+\gamma)$-Carleson measure. Let $d\nu(t)=\frac{1}{1-t}d\mu(t)$, then $\nu$ is a $(1/p+1/q'+\gamma)$-Carleson. And set $s = 1 + p/q'$, the conjugate exponent of $s$ is $s'=1+q'/p$ and $1/p+1/q'=s/p=s'/q'$. Then by \cite[Theorem 9.4]{6}, $H^p$ is continuously embedded in $L^s(d\nu)$, that is,
\begin{align}\label{4.14}
\left( \int_{[0,1)}|f(t)|^sd\nu(t)\right)^{1/s} \leq C||f||_{H^p}, \quad f\in H^p,
\end{align}
and, by Lemma \ref{Le4.1},

\begin{align}\label{4.15}
  \left( \int_{[0,1)}  \frac{1}{(1-t)^{s'/q'}} d\nu(t)      \right) ^{1/s'}
 < \infty, \qquad g \in H^{q'}.
\end{align}
Using H\"{o}lder's inequality with exponents $s$ and $s'$, (\ref{4.14}) and (\ref{4.15}), we obtain that
\begin{align}
& \left|\int_{[0,1)}\overline{f(t)}(g(t)+tg(t)')d\mu(t) \right| \notag\\
& \leq C ||g||_{H^{q'}}\int_{[0,1)}|f(t)| \left({\frac{1}{(1-t)^{1/q'+1}}} \right)d\mu(t) \notag \\
& = C ||g||_{H^{q'}}\int_{[0,1)}|f(t)| \left({\frac{1}{(1-t)^{1/q'}}} \right)d\nu(t) \notag \\
&  \leq C||g||_{H^{q'}} \left( \int_{[0,1)}|f(t)|^sd\nu(t)\right)^{1/s} \left( \int_{[0,1)}  \frac{1}{(1-t)^{s'/q'}} d\nu(t) \right) ^{1/s'} \notag \\
& \leq C ||g||_{H^{q'}}||f||_{H^p},  \quad f\in H^p,g\in H^{q'}.
\end{align}
Hence, (\ref{4.3}) holds and then it follows that $\mathcal{DH}_\mu$ is a bounded operator from $H^p$ into $H^q$.

Conversely, if  $\mathcal{DH}_\mu$ is a bounded operator from $H^p$ into $H^q$,  then $\mu$ is a $(1/p+1/q'+1)$-Carleson measure. The proof is the same as that of Theorem \ref{4.1}(\romannumeral1). We omit the details here.
\end{proof}

We also find $\mathcal{DH}_\mu$ in $H^p(1\leq p \leq 2)$ have a better conclusion.
\begin{theorem}
Let $1 \leq p \leq 2$ and $\mu$ be a positive Borel measure on $[0,1)$ which satisfies the condition in Theorem \ref{Th3.1}. Then $\mathcal{DH}_\mu$ is a bounded operator in $H^p $ if and only if $\mu$ is a $2$-Carleson measure.
\end{theorem}

\begin{proof} Firstly, if $p=1$, by Theorem 4.1 we obtain that $\mathcal{DH}_\mu$ is a bounded operator in $H^1 $ if and only if $\mu$ is a $2$-Carleson measure.

Next, if $p=2$,  by  Theorem 4.2 we only need to prove that if $\mu$ is a $2$-Carleson measure then $\mathcal{DH}_\mu$ is a bounded operator in $H^2$.

Since  $f(z)=\sum_{k=0}^\infty a_k z^k\in H^2$, we have $||f||_{H^2}^2=\sum_{k=0}^\infty|a_k|^2 $, and when $\mu$ is a $2$-Carleson measure, we have
\begin{align}\label{4.191}
|\mu_{n,k}|=|\mu_{n+k}|\leq \frac{C}{(n+k+1)^2}.
\end{align}

By using classical Hilbert inequality, (\ref{1.1}), and (\ref{4.191}), we obtain that
\begin{align}\label{4.201}
||\mathcal{DH}_\mu(f)||_{H^2}^2=\sum_{n=0}^\infty (n+1)^2 \left|\sum_{k=0}^\infty \mu_{n,k}a_k \right|^2 & \leq  \sum_{n=0}^\infty (n+1)^2 \left(\sum_{k=0}^\infty |\mu_{n,k}||a_k| \right)^2 \notag \\
& \leq C \sum_{n=0}^\infty (n+1)^2 \left(\sum_{k=0}^\infty \frac{|a_k|}{(n+k+1)^2} \right)^2 \notag \\
& \leq C \sum_{n=0}^\infty  \left(\sum_{k=0}^\infty \frac{|a_k|}{(n+k+1)} \right)^2 \notag \\
& \leq C \sum_{k=0}^\infty|a_k|^2= C||f||_{H^2}^2.
\end{align}

Thus $\mathcal{DH}_\mu$ is a bounded operator in $H^2$.

Finally, we shall use complex interpolation to prove our results.  We know that
\begin{align}\label{4.21}
 H^p=(H^2,H^1)_\theta, \quad if \; 1<p<2 \; and \; \theta=\frac{2}{p}-1.
\end{align}
 Using (\ref{4.21}) and Theorem 2.4 of \cite{23}, it follows that $\mathcal{DH}_\mu$ is a bounded operator in $H^p(1\leq p\leq 2)$.
\end{proof}
\begin{conjecture}
 We conjecture that if $\mu$ is a $2$-Carleson measure then $\mathcal{DH}_\mu$ is a bounded operator in $H^p$  for all $ 2< p< \infty$.
\end{conjecture}

\section{ Compactness of  $\mathcal{DH}_\mu$ acting on Hardy spaces}\label{s5}
In this section we characterize the compactness of the
Derivative-Hilbert $\mathcal{DH}_\mu$. We begin with the following Lemma \ref{Le4.4} which is useful to deal with the compactness.

\begin{lemma}\label{Le4.4}
For $0 < p < \infty $ and $0 < q < \infty$. Suppose that $\mathcal{DH}_\mu$ is a bounded operator from $H^p$ into $H^q(resp.,B_q)$. Then $\mathcal{DH}_\mu$ is a compact operator if and only if for any bounded sequence $\{f_n\}$ in $H^p$ which converges uniformly to 0 on every compact subset of  $\mathbb{D}$, we have $\mathcal{DH}_\mu(f_n) \rightarrow 0$ in $H^q(resp.,B_q)$.
\end{lemma}

The proof is similar to that of \cite[Proposition 3.11]{4}. We omit the details.

\begin{theorem}\label{Th4.5}
Suppose $0<p \leq 1$ and let $\mu$ be a positive Borel measure on $[0,1)$ which satisfies the condition in Theorem \ref{Th3.1}.
~\\
\item (\romannumeral1) If   $q>1$, then $\mathcal{DH}_\mu$ is a compact operator from $H^p$ into $H^q$ if and only if $\mu$ is a vanishing $(1/p+1/q'+1)$-Carleson measure.
\item (\romannumeral2) If  $q = 1$ , then $\mathcal{DH}_\mu$ is a compact operator from $H^p$ into $H^1$ if and only if $\mu$ is a vanishing $(1/p+1)$-Carleson measure.
\item (\romannumeral3) If  $0<q <1$ , then $\mathcal{DH}_\mu$ is a compact operator from $H^p$ into $B_q$ if and only if $\mu$ is a vanishing $(1/p+1)$-Carleson measure.
\end{theorem}

\begin{proof}
(\romannumeral1) First consider $q>1$. Suppose that $\mathcal{DH}_\mu$ is a compact operator from  $H^p$ into $H^q$. Let $\{a_n\}\subset (0,1)$ be any sequence with $a_n\rightarrow 1$. We set
$$f_{a_n}(z)=\left(\frac{1-a_n^2}{(1-a_n z)^2} \right)^{1/p}, \quad  z\in \mathbb{D}.$$
Then $f_{a_n}(z) \in H^p$, $\sup_{n \geq 1} ||f_{a_n}||_{H^p} < \infty $ and $f_{a_n}\rightarrow0 $, uniformly on any compact subset of $\mathbb{D}$. Using Lemma \ref{Le4.4} and bearing in mind that $\mathcal{DH}_\mu$ is a compact operator from $H^p$ into $H^q$, we obtain that $\{\mathcal{DH}_\mu(f_{a_n} )\}$ converges to 0 in $H^q$. This and (\ref{4.2}) imply that
\begin{align}\label{4.17}
  & \lim_{n\rightarrow \infty }\int_{[0,1)}\overline{f_{a_n}(t)}(g(t)+tg'(t))d\mu(t) \notag \\
  =&\lim_{n\rightarrow \infty}\int_0^{2\pi}\overline{\mathcal{DH}_\mu(f_{a_n})(e^{i\theta})}g(e^{i\theta})d\theta=0, \quad g\in H^{q'}.
\end{align}
Now we wet
$$g_{a_n}(z)=\left(\frac{1-a_n^2}{(1-a_n z)^2} \right)^{1/q'}, \quad  z\in \mathbb{D}.$$
It is obvious find that $g\in H^{q'}$. For every $n$, fix $r\in (a_n,1)$. Thus,
\begin{align}
&\int_{[0,1)}\overline{f_{a_n}(t)}(g_{a_n}(t)+tg'_{a_n}(t))d\mu(t) \notag \\
& \geq  C \int_{[a_n,1)} \left(\frac{1-a_n^2}{(1-a_nt)^2} \right)^{1/p} \left(\left(\frac{1-a_n^2}{(1-a_nt)^2}\right)^{1/{q'}}+\frac{2t^2}{q'}
\left(\frac{1-a_n^2}{(1-a_nt)^{2+{q'}}} \right)^{1/{q'}}  \right)  d\mu(t)\notag \\
    & \geq \frac{C}{(1-a_n^2)^{1/p+1/{q'}+1}} \mu([a_n,1)).  \notag
\end{align}

By (\ref{4.17}) and the fact $\{a_n\}\subset(0,1)$ is a sequence with $a_n\rightarrow 1$, as $ n\rightarrow \infty$, we obtain  that
$$\lim_{a\rightarrow1^-}\frac{1}{(1-a_n^2)^{1/p+1/{q'}+1}} \mu([a_n,1))=0.$$
Thus $\mu$ is a vanishing $(1/p+1/q'+1)$-Carleson measure.

On the other hand, suppose that $\mu$ is a vanishing $(1/p+1/q'+1)$-Carleson measure. Let $\{f_n\}_{n=1}^\infty $ be a sequence of $H^p$ functions with $\sup_{n\geq 1} || f_n||_{ H^p} < \infty$ and such that $\{f_n\} \rightarrow 0$, uniformly on any compact subset of $\mathbb{D}$. Then by Lemma \ref{Le4.4}, it is enough to prove that $\{\mathcal{DH}_\mu(f_n )\}\rightarrow0$ in $H^q$.

Taking $g\in H^{q'}$ and $r\in [0,1)$, we have
\begin{align}
   & \int_{[0,1)}|f_n(t)||(g(t)+tg'(t)|d\mu(t) \notag \\
   & = \int_{[0,r)}|f_n(t)||(g(t)+tg'(t)|d\mu(t)+\int_{[r,1)}|f_n(t)||(g(t)+tg'(t)|d\mu(t). \notag
\end{align}
Then $\displaystyle\int_{[0,r)}|f_n(t)||(g(t)+tg'(t)|d\mu(t)$ tends to $0$ as $\{f_n\}\rightarrow 0$ uniformly on any compact subset of $\mathbb{D}$.

And by conclusion in the proof of the boundedness in Theorem \ref{Th4.1}(\romannumeral1), let $d\nu(t)=\frac{1}{(1-t)^{1/q'+1}}d\mu(t)$. We know that $\nu$ is a vanishing $1/p$-Carleson. Then it implies that
\begin{align}
\int_{[r,1)}|f_n(t)||(g(t)+tg'(t)|d\mu(t)  & \leq C ||g||_{H^{q'}} \int_{[0,1)}|f_n(t)|d\nu_r(t) \notag \\
  & \leq C \mathcal{N}(\nu_r)||g||_{H^{q'}}||f_n||_{H^p} \leq C \mathcal{N}(\nu_r)||g||_{H^{q'}}.
  \end{align}
It also tends to $0$ by (\ref{2.1}). Thus
\begin{align}
   &\lim_{n\rightarrow\infty} \left| \int_0^{2\pi}\overline{\mathcal{DH}_\mu(f_{n})(e^{i\theta})}g(e^{i\theta})d\theta \right|
    \notag \\
  =& \lim_{n\rightarrow \infty }\int_{[0,1)}|f_{n}(t)||(g(t)+tg'(t))|d\mu(t)=0, \quad for\; all\;g\in H^{q'}. \notag
\end{align}
It means $\mathcal{DH}_\mu(f_{n})\rightarrow0\; in \; H^q$, by Lemma \ref{Le4.4} we obtain $\mathcal{DH}_\mu$ is a compact operator from $H^p$ into $H^q$.

(\romannumeral2) Let $q=1$. Suppose that $\mathcal{DH}_\mu$ is a compact operator from $H^p$ into $H^1$. Let $\{a_n\}\subset (0,1)$ be any sequence with $a_n\rightarrow 1$ and $f_{a_n}$ defines like in  (\romannumeral1). Lemma \ref{Le4.4} implies that $\{\mathcal{DH}_\mu(f_{a_n} )\}$ converges to 0 in $H^1$. Then we have
\begin{align}\label{4.19}
 & \lim_{n\rightarrow \infty }\int_{[0,1)}\overline{f_{a_n}(t)}(g(rt)+rtg'(rt))d\mu(t) \notag \\
  =&\lim_{n\rightarrow \infty}\int_0^{2\pi}\overline{\mathcal{DH}_\mu(f_{a_n})(re^{i\theta})}g(e^{i\theta})d\theta=0, \quad g\in VMOA.
\end{align}
Set
$$g_{a_n}(z)=\log \frac{e}{1-a_nz}.$$
It is well know that $g\in VMOA$. For $r\in (a_n,1)$, we deduce that
\begin{align}\label{5344}
  & \int_{[0,1)}\overline{f_{a_n}(t)}(g(rt)+rtg'(rt))d\mu(t) \notag \\
  & \geq C \int_{[a_n,1)} \left(\frac{1-a_n^2}{(1-a_nt)^2} \right)^{1/p} \left(\log \frac{e}{1-a_n rt}+\frac{a_n rt}{1-a_n rt}  \right)d\mu(t) \notag \\
  & \geq  \frac{C}{(1-a_n)^{1/p+1}} \mu([a_n,1)). \notag
\end{align}
Letting $a_n\rightarrow1^-$ as $n\rightarrow \infty$, we have
$$\lim_{a\rightarrow1^-}\frac{1}{(1-a_n^2)^{1/p+1}} \mu([a_n,1))=0.$$

we can obtain that $\mu$ is a vanishing $(1/p+1)$-Carleson measure.

On the other hand, suppose that $\mu$ is a vanishing $(1/p+1)$-Carleson measure. Let $d\nu(t)=(1-t)^{-1}d\mu(t)$, we know that $\nu$ is a vanishing $1/p$-Carleson.
Let $\{f_n\}_{n=1}^\infty $ be a sequence of $H^p$ functions with $\sup_{n\geq 1} || f_n||_{ H^p} < \infty$ and such that $\{f_n\} \rightarrow 0$, uniformly on any compact subset of $\mathbb{D}$. Then by Lemma \ref{Le4.4}, it is enough to prove that $\{\mathcal{DH}_\mu(f_n )\}\rightarrow0$ in $H^1$.
For every $g \in VMOA, 0 <r< 1$ , using (\ref{3.2}) and (\ref{4.12}), we deduce that
\begin{align}
   & \int_{[0,1)}|f_n(t)||(g(t)+tg'(t)|d\mu(t) \notag \\
   & = \int_{[0,r)}|f_n(t)||(g(t)+tg'(t)|d\mu(t)+\int_{[r,1)}|f_n(t)||(g(t)+tg'(t)|d\mu(t). \notag
\end{align}
Then $\displaystyle\int_{[0,r)}|f(t)||(g(t)+tg'(t)|d\mu(t)$ tends to 0 as $\{f_n\}\rightarrow 0$ uniformly on any compact subset of $\mathbb{D}$. For second term, arguing as in the proof of the boundedness in Theorem \ref{Th4.1} (\romannumeral2), we obtain that
\begin{align}
\int_{[r,1)}|f_n(t)||(g(t)+tg'(t)|d\mu(t)  &  \leq C \|g\|_{BMOA}\int_{[0,1)}|f_n(t)|\left({\log \frac{1}{1-t}}+\frac{t}{1-t} \right)d\mu(t) \notag \\
& \leq C\|g\|_{BMOA} \int_{[0,1)}|f_n(t)|d\nu_r(t) \notag \\
  & \leq C \mathcal{N}(\nu_r)\|g\|_{BMOA}\|f_n\|_{H^p}, \notag \\
  & \leq C \mathcal{N}(\nu_r)\|g\|_{BMOA},\, g\in VMOA,
  \end{align}
it also tends to 0 by (\ref{2.1}). Thus
\begin{align}
   &\lim_{n\rightarrow\infty} \left| \int_0^{2\pi}\overline{\mathcal{DH}_\mu(f_{n})(e^{i\theta})}g(e^{i\theta})d\theta \right|
    \notag \\
  =& \lim_{n\rightarrow \infty }\int_{[0,1)}|f_{n}(t)||(g(t)+tg'(t))|d\mu(t)=0, \quad for\; all\;g\in VMOA. \notag
\end{align}
It means $\mathcal{DH}_\mu(f_{n})\rightarrow0 \;in \; H^1$, by Lemma \ref{Le4.4} we obtain $\mathcal{DH}_\mu$ is a compact operator from $H^p$ into $H^1$.

(\romannumeral3) The proof is the same as that of Theorem \ref{Th4.1}(\romannumeral3) and Theorem \ref{Th4.5}(1). We omit the details here.
\end{proof}

Finally, we consider the situation of $ p> 1$, characterize those measures $\mu$ for which $\mathcal{DH}_\mu$ is a compact operator from $H^p$ into $H^q$, and give sufficient and necessary conditions respectively.

\begin{theorem}\label{Th4.6}
 Let $1<p\leq q<\infty$  and $\mu$ be a positive Borel measure on $[0,1)$ which satisfies the condition in Theorem \ref{Th3.1}.
\item (\romannumeral1) If $\mu$ is a vanishing $(1/p+1/q'+1+\gamma)$-Carleson measure for any $\gamma>0$, then $\mathcal{DH}_\mu$ is a compact operator from $H^p$ into $H^q$.
\item (\romannumeral2) If $\mathcal{DH}_\mu$ is a compact operator from $H^p$ into $H^q$, then $\mu$ is a vanishing $(1/p+1/q'+1)$-Carleson measure.
\end{theorem}

\begin{proof}
(\romannumeral1)The proof is the same as that of Theorem \ref{Th4.5}(\romannumeral1).We omit the details here.

(\romannumeral2)The proof is similar to that of Theorem \ref{Th4.2}(\romannumeral2) and Theorem \ref{Th4.5}(\romannumeral1). We omit the details here.
\end{proof}

Similarly, $\mathcal{DH}_\mu$ in $H^p(1\leq p \leq 2)$ also have a better conclusion.
\begin{theorem}
 Let $1 \leq p \leq 2$ and $\mu$ be a positive Borel measure on $[0,1)$ which
satisfies the condition in Theorem \ref{Th3.1}. Then $\mathcal{DH}_\mu$ is a compact operator in $H^p$ if and only if $\mu$ is a vanishing 2-Carleson measure.
\end{theorem}

\begin{proof}
Firstly, let $p=1$, we know that $\mathcal{DH}_\mu$ is a compact operator in $H^1$ if and only if $\mu$ is a vanishing 2-Carleson measure by Theorem 5.1.

Next, let $p=2$, by  Theorem 5.2,  We only need to prove
if $\mu$ is a vanishing 2-Carleson measure then $\mathcal{DH}_\mu$ is a compact operator in $H^2$.

Assume that $\mu$  is a vanishing 2-Carleson measure and let $\{ f_j\}$ be a sequence of functions in $H^2$ with $ ||f_j||_{H^2} \leq 1$, for all $j$, and such that $f_j \rightarrow 0$, uniformly on compact subsets of $\mathbb{D}$.
Since $\mu$  is a vanishing 2-Carleson measure then $\mu_{n+k} = o (\frac{1}{(n+k+1)^2} )$, as $n \rightarrow \infty $. Say
$$\mu_{n,k}=\mu_{n+k}=\frac{\varepsilon_{n}}{(n+k+1)^2},\quad n=0,1,2,\ldots.$$
Then $\{\varepsilon_n\}\rightarrow0$. Say that, for every $j$,
$$f_j(z)=\sum_{k=0}^\infty a_k^{(j)}z^k,\quad z\in \mathbb{D}.$$
By using the classical Hilbert inequality, we have
\begin{align}
\sum_{n=0}^\infty \left|\sum_{k=0}^\infty \frac{a^{(j)}_k}{n+k+1} \right|^2 \leq \pi^2\sum_{k=0}^\infty |{a^{(j)}_k}|^2\leq \pi^2.
\end{align}
Take $\varepsilon>0$ and next take a natural number $N$ such that
$$n\geq N \quad\Rightarrow  \quad \varepsilon_n^2 < \frac{\varepsilon}{2\pi^2}.$$
We have
\begin{align}\label{4.2321}
||\mathcal{DH}_\mu(f_j)||_{H^2}^2& =\sum_{n=0}^\infty (n+1)^2 \left|\sum_{k=0}^\infty \mu_{n,k}a^{(j)}_k \right|^2 \notag \\
& = \sum_{n=0}^N (n+1)^2 \left|\sum_{k=0}^\infty \mu_{n,k}a^{(j)}_k \right|^2+\sum_{n=N+1}^\infty (n+1)^2 \left|\sum_{k=0}^\infty \mu_{n,k}a^{(j)}_k \right|^2 \notag \\
& \leq \sum_{n=0}^N (n+1)^2 \left|\sum_{k=0}^\infty \mu_{n,k}a^{(j)}_k \right|^2+\sum_{n=0}^\infty (n+1)^2 \left|\sum_{k=0}^\infty \frac{\varepsilon_na^{(j)}_k}{(n+k+1)^2} \right|^2 \notag \\
& \leq \sum_{n=0}^N (n+1)^2 \left|\sum_{k=0}^\infty \mu_{n,k}a^{(j)}_k \right|^2+\frac{\varepsilon}{2\pi^2} \sum_{n=0}^\infty \left|\sum_{k=0}^\infty \frac{a^{(j)}_k}{n+k+1} \right|^2 \notag\\
&\leq \sum_{n=0}^N (n+1)^2 \left|\sum_{k=0}^\infty \mu_{n,k}a^{(j)}_k \right|^2 +\frac{\varepsilon}{2}.
\end{align}
Now, since $f_j \rightarrow 0$, uniformly on compact subsets of $\mathbb{D}$, it follows that
$$\sum_{n=0}^N (n+1)^2 \left|\sum_{k=0}^\infty \mu_{n,k}a^{(j)}_k \right|^2  \rightarrow0, \quad as \; j \rightarrow\infty.$$
Then it follows that that there exist $j_0 \in N$ such that $||\mathcal{DH}_\mu(f_j)||_{H^2}^2 <\varepsilon $ for all $j \geq j_0$. So, we have proved that $||\mathcal{DH}_\mu(f_j)||_{H^2}^2 \rightarrow 0 $. The compactness of $\mathcal{DH}_\mu$ on $H^2$ follows.

Since we have prove that when $p=1$, the compactness of $\mathcal{DH}_\mu$ on $H^1$. To deal with the cases $1 < p < 2$, we use again complex interpolation. Let  $1 < p < 2$ and $\mu$ be a vanishing 2-Carleson measure. Recall that
\begin{align}
 H^p=(H^2,H^1)_\theta, \quad if \; 1<p<2 \; and \; \theta=\frac{2}{p}-1. \notag
\end{align}

We have also that if $2 < s < \infty $ then
\begin{align}
 H^2=(H^s,H^1)_\alpha.  \notag
\end{align}
for a certain $\alpha \in (0, 1)$, namely, $\alpha = ( 1/2-1/s ) / (1 -1/s )$. Since $H^2$ is reflexive, and $\mathcal{DH}_\mu$ is compact from $H^2$ into itself and from $H^1$ into itself, Theorem 10 of \cite{24} gives that  $\mathcal{DH}_\mu$ is a compact operator in $H^p(1\leq p\leq2)$.
\end{proof}

\textbf{Conflicts of Interest}

The authors declare that there are no conflicts of interest regarding the publication of this paper.\\
\textbf{Availability of data and material}

The authors declare that all data and material in this paper are available.

\end{document}